\documentclass[12pt]{article}
\usepackage{amsmath,amsfonts,amssymb,amsthm}

\theoremstyle{plain}
\newtheorem{theorem}{Theorem}[section]
\newtheorem{corollary}[theorem]{Corollary}
\newtheorem{lemma}[theorem]{Lemma}
\newtheorem{proposition}[theorem]{Proposition}

\theoremstyle{definition}

\newtheorem{example}[theorem]{Example}
\newtheorem{remark}[theorem]{Remark}

\newtheorem*{openproblem}{Open Problem}

\numberwithin{equation}{section}

\newcommand{\R}{{\mathbb R}}
\newcommand{\N}{{\mathbb N}}

\providecommand{\vint}[1]{\mathchoice
          {\mathop{\vrule width 5pt height 3 pt depth -2.5pt
                  \kern -9pt \kern 1pt\intop}\nolimits_{\kern -5pt{#1}}}
          {\mathop{\vrule width 5pt height 3 pt depth -2.6pt
                  \kern -6pt \intop}\nolimits_{\kern -3pt{#1}}}
          {\mathop{\vrule width 5pt height 3 pt depth -2.6pt
                  \kern -6pt \intop}\nolimits_{\kern -3pt{#1}}}
          {\mathop{\vrule width 5pt height 3 pt depth -2.6pt
                  \kern -6pt \intop}\nolimits_{\kern -3pt{#1}}}}

\newcommand{\eps}{\varepsilon}
\newcommand{\loc}{\mathrm{loc}}

\newcommand{\BV}{\mathrm{BV}}
\newcommand{\liploc}{\mathrm{Lip}_{\mathrm{loc}}}

\newcommand{\ch}{\text{\raise 1.3pt \hbox{$\chi$}\kern-0.2pt}}

\DeclareMathOperator{\capa}{Cap}

\DeclareMathOperator{\dist}{dist}
\DeclareMathOperator{\diam}{diam}

\DeclareMathOperator{\Lip}{Lip}

\DeclareMathOperator{\supp}{supp}

\DeclareMathOperator{\inte}{int}

\def\XXint#1#2#3{{\setbox0=\hbox{$#1{#2#3}{\int}$}
\vcenter{\hbox{$#2#3$}}\kern-.5\wd0}}

\begin{document}
\title{Strong approximation of sets \\
of  finite perimeter in metric spaces
\footnote{
\hfill \break {\it Keywords\,}: metric measure space, doubling measure, Poincar\'e inequality, bounded variation, finite perimeter,
$\BV$ norm, strong approximation, measure theoretic boundary
}}
\author{Panu Lahti \\
\\
\noindent Department of Mathematical Sciences\\
4199 French Hall West\\
University of Cincinnati\\
2815 Commons Way\\
Cincinnati, OH 45221-0025\\
E-mail: {\tt lahtipk@ucmail.uc.edu}
}
\maketitle

\begin{abstract}
In the setting of a metric space equipped with a doubling measure that supports a Poincar\'e inequality, we show that any set of finite perimeter can be approximated in the $\BV$ norm by a set whose topological and measure theoretic boundaries almost coincide.
This result appears to be new even in the Euclidean setting. The work relies on a quasicontinuity-type result for $\BV$ functions proved by Lahti and Shanmugalingam (2016, \cite{LaSh}).
\end{abstract}

\section{Introduction}

It is well known in the Euclidean setting that a set of finite perimeter can be approximated in a weak sense by sets with smooth boundaries, see e.g. \cite[Theorem 3.42]{AFP}. In the setting of a much more general metric space, it was shown in \cite{ADG} that a set of finite perimeter can be approximated in the $L^1$-sense by sets whose boundaries are sufficiently regular that their Minkowski contents converge to the perimeter of the set.

On the other hand, fairly little seems to be known about approximating sets of finite perimeter in the \emph{$\BV$ norm}.
In the Euclidean setting, this type of result was given in \cite[Theorem 3.1]{TQG}, where it was shown that given a set $E$ of finite perimeter in an open set $\Omega$ and $\eps>0$,
the set $E$ can be approximated in the $\BV(\Omega)$-norm by a set $F$ whose boundary $\partial F\cap \Omega$ is contained in a finite union of $C^1$ hypersurfaces, and so that $\mathcal H^{n-1}(\Omega\cap \partial F\setminus \partial^*F)<\eps$, where $\partial^*F$ is the measure theoretic boundary.

In this paper we show a similar result in a metric space equipped with a doubling measure that supports a Poincar\'e inequality. More precisely, if $\Omega\subset X$ is an open set and $E\subset X$ is a set of finite perimeter in $\Omega$, and $\eps>0$, we show that there exists a set $F\subset X$ with
\[
\Vert \ch_F-\ch_E\Vert_{\BV(\Omega)}<\eps\quad\textrm{and}\quad \mathcal H(\Omega\cap \partial F\setminus \partial^*F)=0,
\]
where $\mathcal H$ is the codimension $1$ Hausdorff measure. This is given in Theorem \ref{thm:approximation for sets of finite perimeter}.
This is a partial generalization of  \cite[Theorem 3.1]{TQG} to the metric setting, and in fact a partial improvement already in the Euclidean setting, since we are able to show that $\mathcal H(\Omega\cap \partial F\setminus \partial^*F)$ is zero instead of just being small. This is a fairly strong regularity requirement on the boundary, since in general the topological boundary of a set of finite perimeter can be much bigger than the measure theoretic boundary, see Example \ref{ex:enlarged rationals}. The proof of 
 Theorem \ref{thm:approximation for sets of finite perimeter} is heavily based on a quasicontinuity-type result for $\BV$ functions given in \cite[Theorem 1.1]{LaSh}.

\paragraph{Acknowledgments.} The research was
funded by a grant from the Finnish Cultural Foundation. Part of the research was conducted during a visit to the University of Oxford. The author wishes to thank this institution for its hospitality, and professor Jan Kristensen for pointing out the reference \cite{TQG}.

\section{Notation and background}

In this section we introduce the necessary notation and assumptions.

In this paper, $(X,d,\mu)$ is a complete metric space equipped
with a Borel regular outer measure $\mu$ satisfying a doubling property, that is,
there is a constant $C_d\ge 1$ such that
\[
0<\mu(B(x,2r))\leq C_d\,\mu(B(x,r))<\infty
\]
for every ball $B=B(x,r)$ with center $x\in X$ and radius $r>0$.
Sometimes we abbreviate $\alpha B(x,r):=B(x,\alpha r)$, $\alpha>0$.
We assume that $X$ consists of at least two points. By iterating the doubling condition, we obtain that for any $x\in X$ and $y\in B(x,R)$ with $0<r\le R<\infty$, we have
\begin{equation}\label{eq:homogenous dimension}
\frac{\mu(B(y,r))}{\mu(B(x,R))}\ge \frac{1}{C}\left(\frac{r}{R}\right)^{Q},
\end{equation}
where $C\ge 1$ and $Q> 0$ only depend on the doubling constant $C_d$.
In general, $C\ge 1$ will denote a constant whose particular value is not important for the purposes of this
paper, and might differ between
each occurrence. When we want to specify that a constant $C$
depends on the parameters $a,b, \ldots,$ we write $C=C(a,b,\ldots)$. Unless otherwise specified, all constants only 
depend on the space $X$, more precisely on the doubling constant $C_d$, the constants $C_P,\lambda$ associated with the Poincar\'e inequality defined below, and $\diam(X)$.

A complete metric space with a doubling measure is proper,
that is, closed and bounded sets are compact. Since $X$ is proper, for any open set $\Omega\subset X$
we define $\liploc(\Omega)$ to be the space of
functions that are Lipschitz in every $\Omega'\Subset\Omega$.
Here $\Omega'\Subset\Omega$ means that $\Omega'$ is open and that $\overline{\Omega'}$ is a
compact subset of $\Omega$. Other local spaces of functions are defined similarly.

For any set $A\subset X$ and $0<R<\infty$, the restricted spherical Hausdorff content
of codimension $1$ is defined by
\begin{equation}\label{eq:definition of Hausdorff content}
\mathcal{H}_{R}(A):=\inf\left\{ \sum_{i=1}^{\infty}
  \frac{\mu(B(x_{i},r_{i}))}{r_{i}}:\,A\subset\bigcup_{i=1}^{\infty}B(x_{i},r_{i}),\,r_{i}\le R\right\}.
\end{equation}
We define the above also for $R=\infty$ by requiring $r_i<\infty$. 
The codimension $1$ Hausdorff measure of a set $A\subset X$ is given by
\begin{equation*}%\label{eq:codimension 1 Hausdorff measure}
  \mathcal{H}(A):=\lim_{R\rightarrow 0}\mathcal{H}_{R}(A).
\end{equation*}
For any outer measure $\nu$ on $X$, the codimension $1$ Minkowski content of a set $A\subset X$ is defined by
\[
\nu^{+}(A):=\liminf_{R\to 0}\frac{\nu\left(\bigcup_{x\in A}B(x,R)\right)}{2R}.
\]

The measure theoretic boundary $\partial^{*}E$ of a set $E\subset X$ is the set of points $x\in X$
at which both $E$ and its complement have positive upper density, i.e.
\[
\limsup_{r\to 0}\frac{\mu(B(x,r)\cap E)}{\mu(B(x,r))}>0\quad\;
  \textrm{and}\quad\;\limsup_{r\to 0}\frac{\mu(B(x,r)\setminus E)}{\mu(B(x,r))}>0.
\]
The measure theoretic interior and exterior of $E$ are defined respectively by
\begin{equation}\label{eq:definition of measure theoretic interior}
I_E:=\left\{x\in X:\,\lim_{r\to 0}\frac{\mu(B(x,r)\setminus E)}{\mu(B(x,r))}=0\right\}
\end{equation}
and
\begin{equation}\label{eq:definition of measure theoretic exterior}
O_E:=\left\{x\in X:\,\lim_{r\to 0}\frac{\mu(B(x,r)\cap E)}{\mu(B(x,r))}=0\right\}.
\end{equation}
A curve $\gamma$ is a rectifiable continuous mapping from a compact interval
into $X$.
A nonnegative Borel function $g$ on $X$ is an upper gradient 
of an extended real-valued function $u$
on $X$ if for all curves $\gamma$ on $X$, we have
\begin{equation}\label{eq:definition of upper gradient}
|u(x)-u(y)|\le \int_\gamma g\,ds,
\end{equation}
where $x$ and $y$ are the end points of $\gamma$. We interpret $|u(x)-u(y)|=\infty$ whenever  
at least one of $|u(x)|$, $|u(y)|$ is infinite.
Of course, by replacing $X$ with a set $A\subset X$ and considering curves $\gamma$ in $A$, we can talk about a function $g$ being an upper gradient of $u$ in $A$.
We define the local Lipschitz constant of a locally Lipschitz function $u\in\liploc(X)$ by
\[
\Lip u(x):=\limsup_{r\to 0}\sup_{y\in B(x,r)\setminus \{x\}}\frac{|u(y)-u(x)|}{d(y,x)}.
\]
Then $\Lip u$ is an upper gradient of $u$, see e.g. \cite[Proposition 1.11]{Che}.
Upper gradients were originally introduced in~\cite{HK}.

If $g$ is a nonnegative $\mu$-measurable function on $X$
and (\ref{eq:definition of upper gradient}) holds for $1$-almost every curve,
we say that $g$ is a $1$-weak upper gradient of~$u$. 
A property holds for $1$-almost every curve
if it fails only for a curve family with zero $1$-modulus. 
A family $\Gamma$ of curves is of zero $1$-modulus if there is a 
nonnegative Borel function $\rho\in L^1(X)$ such that 
for all curves $\gamma\in\Gamma$, the curve integral $\int_\gamma \rho\,ds$ is infinite.

Given an open set $\Omega\subset X$, we consider the following norm
\[
\Vert u\Vert_{N^{1,1}(\Omega)}:=\Vert u\Vert_{L^1(\Omega)}+\inf \Vert g\Vert_{L^1(\Omega)},
\]
with the infimum taken over all $1$-weak upper gradients $g$ of $u$ in $\Omega$.
The substitute for the Sobolev space $W^{1,1}(\Omega)$ in the metric setting is the Newton-Sobolev space
\[
N^{1,1}(\Omega):=\{u:\|u\|_{N^{1,1}(\Omega)}<\infty\}.
\]
It is known that for any $u\in N_{\loc}^{1,1}(\Omega)$, there exists a minimal $1$-weak
upper gradient, denoted by $g_u$, that satisfies $g_{u}\le g$ 
$\mu$-almost everywhere in $\Omega$, for any $1$-weak upper gradient $g\in L_{\loc}^{1}(\Omega)$
 of $u$ in $\Omega$, see \cite[Theorem 2.25]{BB}.
For more on Newton-Sobolev spaces, we refer to~\cite{S, BB, HKST}.

Next we recall the definition and basic properties of functions
of bounded variation on metric spaces, see \cite{M}. See also e.g. \cite{AFP, EvaG92, Giu84, Zie89} for the classical 
theory in the Euclidean setting.
For $u\in L^1_{\loc}(X)$, we define the total variation of $u$ in $X$ to be 
\[
\|Du\|(X):=\inf\left\{\liminf_{i\to\infty}\int_X g_{u_i}\,d\mu:\, u_i\in \Lip_{\loc}(X),\, u_i\to u\textrm{ in } L^1_{\loc}(X)\right\},
\]
where each $g_{u_i}$ is an upper gradient of $u_i$.
We say that a function $u\in L^1(X)$ is of bounded variation, 
and denote $u\in\BV(X)$, if $\|Du\|(X)<\infty$.
By replacing $X$ with an open set $\Omega\subset X$ in the definition of the total variation, we can define $\|Du\|(\Omega)$.
For an arbitrary set $A\subset X$, we define
\[
\|Du\|(A)=\inf\{\|Du\|(\Omega):\, A\subset\Omega,\,\Omega\subset X
\text{ is open}\}.
\]
If $u\in\BV(\Omega)$, $\|Du\|(\cdot)$ is a finite Radon measure on $\Omega$ by \cite[Theorem 3.4]{M}.
The $\BV$ norm is defined by
\[
\Vert u\Vert_{\BV(\Omega)}:=\Vert u\Vert_{L^1(\Omega)}+\Vert Du\Vert(\Omega).
\]
A $\mu$-measurable set $E\subset X$ is said to be of finite perimeter if $\|D\ch_E\|(X)<\infty$, where $\ch_E$ is the characteristic function of $E$.
The perimeter of $E$ in $\Omega$ is also denoted by
\[
P(E,\Omega):=\|D\ch_E\|(\Omega).
\]
Similarly as above, if $P(E,\Omega)<\infty$, then $P(E,\cdot)$ is a finite Radon measure
on $\Omega$.
For any Borel sets $E_1,E_2\subset X$ we have by \cite[Proposition 4.7]{M}
\begin{equation}\label{eq:Caccioppoli sets form an algebra}
P(E_1\cup E_2,X)\le P(E_1,X)+P(E_2,X). 
\end{equation}
Similarly it can be shown that if $\Omega\subset X$ is an open set and $u,v\in L^1_{\loc}(\Omega)$, then
\begin{equation}\label{eq:subadditivity}
\Vert D(u+v)\Vert(\Omega)\le \Vert Du\Vert(\Omega)+\Vert Dv\Vert(\Omega).
\end{equation}
We have the following coarea formula from~\cite[Proposition 4.2]{M}: if $\Omega\subset X$ is an open set and $u\in L^1_{\loc}(\Omega)$, then
\begin{equation}\label{eq:coarea}
\|Du\|(\Omega)=\int_{-\infty}^{\infty}P(\{u>t\},\Omega)\,dt.
\end{equation}
If $\Vert Du\Vert(\Omega)<\infty$, the above is true with $\Omega$ replaced by any Borel set $A\subset\Omega$.

We will assume throughout that $X$ supports a $(1,1)$-Poincar\'e inequality,
meaning that there exist constants $C_P\ge 1$ and $\lambda \ge 1$ such that for every
ball $B(x,r)$, every $u\in L^1_{\loc}(X)$,
and every upper gradient $g$ of $u$, we have 
\[
\vint{B(x,r)}|u-u_{B(x,r)}|\, d\mu 
\le C_P r\vint{B(x,\lambda r)}g\,d\mu,
\]
where 
\[
u_{B(x,r)}:=\vint{B(x,r)}u\,d\mu :=\frac 1{\mu(B(x,r))}\int_{B(x,r)}u\,d\mu.
\]
%By applying the Poincar\'e inequality to approximating locally Lipschitz functions in the definition of the total variation, 
%we get the following $(1,1)$-Poincar\'e inequality for $\BV$ functions. There exists a constant $C$
%such that for every ball $B(x,r)$ and every 
%$u\in L^1_{\loc}(X)$, we have
%\[
%\vint{B(x,r)}|u-u_{B(x,r)}|\,d\mu
%\le Cr\, \frac{\Vert Du\Vert (B(x,\lambda r))}{\mu(B(x,\lambda r))}.
%\]
%For $\mu$-measurable sets $E\subset X$, the above can be written as
%\begin{equation}\label{eq:relative isoperimetric inequality}
%\min\{\mu(B(x,r)\cap E),\,\mu(B(x,r)\setminus E)\}\le CrP(E,B(x,\lambda r)).
%\end{equation}

The $1$-capacity of a set $A\subset X$ is given by
\[
\capa_1(A):=\inf \Vert u\Vert_{N^{1,1}(X)},
\]
where the infimum is taken over all functions $u\in N^{1,1}(X)$ such that $u\ge 1$ in $A$.
For basic properties satisfied by the $1$-capacity, such as monotonicity and countable subadditivity, see e.g. \cite{BB}.

Given a set of finite perimeter $E\subset X$, for $\mathcal H$-almost every $x\in \partial^*E$ we have
\begin{equation}\label{eq:definition of gamma}
\gamma \le \liminf_{r\to 0} \frac{\mu(E\cap B(x,r))}{\mu(B(x,r))} \le \limsup_{r\to 0} \frac{\mu(E\cap B(x,r))}{\mu(B(x,r))}\le 1-\gamma
\end{equation}
where $\gamma \in (0,1/2]$ only depends on the doubling constant and the constants in the Poincar\'e inequality, 
see~\cite[Theorem 5.4]{A1}.
For an open set $\Omega\subset X$ and a $\mu$-measurable set $E\subset X$
with $P(E,\Omega)<\infty$, we have for any Borel set $A\subset\Omega$
\begin{equation}\label{eq:def of theta}
P(E,A)=\int_{\partial^{*}E\cap A}\theta_E\,d\mathcal H,
\end{equation}
where
$\theta_E\colon X\to [\alpha,C_d]$ with $\alpha=\alpha(C_d,C_P,\lambda)>0$, see \cite[Theorem 5.3]{A1} 
and \cite[Theorem 4.6]{AMP}.

The lower and upper approximate limits of a $\mu$-measurable function $u$ on $X$ are defined respectively by
\begin{equation}\label{eq:lower approximate limit}
u^{\wedge}(x):
=\sup\left\{t\in\overline\R:\,\lim_{r\to 0}\frac{\mu(B(x,r)\cap\{u<t\})}{\mu(B(x,r))}=0\right\}
\end{equation}
and
\begin{equation}\label{eq:upper approximate limit}
u^{\vee}(x):
=\inf\left\{t\in\overline\R:\,\lim_{r\to 0}\frac{\mu(B(x,r)\cap\{u>t\})}{\mu(B(x,r))}=0\right\}.
\end{equation}

Note that we understand $\BV$ functions to be $\mu$-equivalence classes. To consider 
continuity properties, we need to consider the pointwise representatives $u^{\wedge}$ and
$u^{\vee}$. We also define the representative
\begin{equation}\label{eq:precise representative}
\widetilde{u}:=(u^{\wedge}+u^{\vee})/2.
\end{equation}

\section{Preliminary measure theoretic results}

In this section we discuss some measure theoretic results that will be needed in the proof of our main result.

First we note that the following coarea inequality holds.

\begin{lemma}\label{lem:coarea inequality}
If $U\subset X$ is an open set and $w\in\liploc(U)$, then
\[
\int_{-\infty}^{\infty}\mathcal H(U\cap \partial\{w>t\})\,dt\le C_{\textrm{co}}\int_{U}\Lip w\,d\mu,
\]
where $C_{\textrm{co}}=C_{\textrm{co}}(C_d)$.
\end{lemma}
\begin{proof}
By \cite[Proposition 3.5]{KoLa} (which is based on \cite[Lemma 3.1]{BH}), the following coarea inequality holds:
if $\nu$ is a positive Radon measure of finite mass and
$u\in\Lip(X)$ is bounded, then
\[
\int_{-\infty}^{\infty}\nu^+(\partial\{u>t\})\,dt\le \int_{X}\Lip u\,d\nu.
\]
Choose $U''\Subset U'\Subset U$ and let $\nu:=\mu|_{U'}$, so that $\nu$ is of finite mass.
Define a function $u:=w$ in $U'$, so that $u\in\Lip(U')$, and extend it to a bounded function $u\in \Lip(X)$. Since $\mathcal H(A)\le C_d^3\mu^+(A)$ for any $A\subset X$ (see e.g. \cite[Proposition 3.12]{KoLa}), we have
\begin{align*}
\frac{1}{C_d^3}\int_{-\infty}^{\infty}\mathcal H(U''\cap \partial\{w>t\})\,dt
&\le\int_{-\infty}^{\infty}\mu^+(U''\cap \partial\{w>t\})\,dt\\
&= \int_{-\infty}^{\infty}\mu^+(U''\cap \partial\{u>t\})\,dt\\
&\le \int_{-\infty}^{\infty}\nu^+(\partial\{u>t\})\,dt\\
&\le  \int_{X}\Lip u\,d\nu
=  \int_{U'}\Lip w\,d\mu.
\end{align*}
By letting $U''\nearrow U$ and using Lebesgue's monotone convergence theorem on both sides, we obtain the result with $C_{\textrm{co}}=C_d^3$.
\end{proof}

Now we can show that the level sets of a locally Lipschitz function have the following
weak regularity.

\begin{proposition}\label{prop:null difference between topological and meas theor bdry}
Let $U\subset X$ be an open set and let $w\in\liploc(U)$. Then
$\mathcal H(U\cap\partial \{w>t\}\setminus \partial^*\{w>t\})=0$ for almost every $t\in\R$.
\end{proposition}

\begin{proof}
Fix $U'\Subset U$. Note that $w\in\Lip(U')\subset \BV(U')$. Let
\[
A:=U'\cap\bigcup_{s\in\R}\partial^*\{w>s\}.
\]
Note that in $U$, $\partial^*\{w>s\}\subset \partial\{w>s\}\subset \{w=s\}$, which are pairwise disjoint sets for distinct values of $s$.
Note also that for any open set $V\subset U'$, denoting the minimal $1$-weak upper gradient of $w$ in $V$ by $g_w$, we have
\begin{equation}\label{eq:Lipschitz constant and variation measure}
\int_V \Lip w\,d\mu\le C \int_V g_w\,d\mu\le C\Vert Dw\Vert(V),
\end{equation}
where the first inequality follows from the fact that $\Lip w\le Cg_w$ $\mu$-almost everywhere, see \cite[Proposition 4.26]{Che} or \cite[Proposition 13.5.2]{HKST}, and the second inequality follows from \cite[Remark 4.7]{HKLL}.

By using the disjointness of the sets $\partial^*\{u>s\}\subset \partial\{u>s\}$ in $U$, Lemma \ref{lem:coarea inequality}, and \eqref{eq:Lipschitz constant and variation measure}, we estimate for any open set $V$ with $ U'\setminus A\subset V\subset U'$
\begin{align*}
\int_{-\infty}^{\infty} \mathcal H(U'\cap\partial \{w>t\}\setminus \partial^*\{w>t\})\,dt
& =\int_{-\infty}^{\infty} \mathcal H(U'\cap\partial \{w>t\}\setminus A)\,dt\\
& \le \int_{-\infty}^{\infty} \mathcal H(V\cap \partial \{w>t\})\,dt\\
& \le \int_V \Lip w\,d\mu\\
& \le C\Vert Dw\Vert(V).
\end{align*}
By taking the infimum of open sets $V$ as above, we get $C\Vert Dw\Vert(U'\setminus A)$ on the right-hand side. By also using the $\BV$ coarea inequality \eqref{eq:coarea} and \eqref{eq:def of theta}, we obtain
\begin{align*}
\int_{-\infty}^{\infty} \mathcal H(U'\cap\partial \{w>t\}\setminus \partial^*\{w>t\})\,dt
& \le C\Vert Dw\Vert(U'\setminus A)\\
& = C\int_{-\infty}^{\infty} P(\{w>t\},U'\setminus A)\,dt\\
& \le C\int_{-\infty}^{\infty} \mathcal H(U'\cap\partial^*\{w>t\}\setminus A)\,dt\\
&=0
\end{align*}
since $U'\cap\partial^*\{w>t\}\setminus A=\emptyset$ for all $t\in\R$.
By exhausting $U$ by sets $U'\Subset U$, we obtain the result.
\end{proof}

Since we are going to work with quasicontinuity-type results,
in the following we prove a few results on how to analyse and manipulate sets of small capacity.

\begin{remark}\label{rmk:capacities and Hausdorff contents}
The $1$-capacity and the Hausdorff contents are closely related:
it follows from~\cite[Theorem~4.3, Theorem~5.1]{HaKi} that
$\capa_1(A)=0$ if and only if $\mathcal{H}(A)=0$. 
For any $R>0$ and $A\subset X$, from the proof of \cite[Lemma 3.4]{KKST3} it follows that
\[
\capa_1(A)\le C(C_d,C_P,\lambda,R)\mathcal H_R(A),
\]
and 
by combining~\cite[Theorem~4.3]{HaKi} and the proof of~\cite[Theorem 5.1]{HaKi}
we obtain that conversely
\[
\mathcal H_{R}(A)\le C(C_d,C_P,\lambda,R)\capa_1(A).
\]
Finally, we note that $\capa_1$ is an outer capacity, meaning that
\[
\capa_1(A)=\inf\{\capa_1(U):\,U\supset A\textrm{ is open}\}
\]
for any $A\subset X$,
see e.g. \cite[Theorem 5.31]{BB}. 
\end{remark}

\begin{lemma}\label{lem:uniform convergence}
Let $A\subset X$ be a Borel set with $\mathcal H(A)<\infty$, and let $\eps>0$. Then there exists an open set $U\supset A$ with $\capa_1(U)\le C\mathcal H(A)+\eps$ such that
\[
r\frac{\mathcal H(A\cap B(x,r))}{\mu(B(x,r))}\to 0\qquad\textrm{as }r\to 0
\]
uniformly for all  $x\in X\setminus U$.
\end{lemma}

\begin{proof}
Inductively, we can pick compact sets $H_i\subset A\setminus \bigcup_{j=1}^{i-1}H_j$ with
\[
\mathcal H\left(A\setminus \bigcup_{j=1}^i H_j\right)<2^{-i}\eps,\quad i\in\N;
\]
see e.g. \cite[Proposition 1.43]{AFP}.
Also pick open sets $U_i\supset H_i$ and a decreasing sequence of numbers $1/5\ge r_1\ge r_2\ge \ldots$ with $\dist(H_i,X\setminus U_i)\ge r_i$ such that
\[
\capa_1(U_i)\le \capa_1(H_i)+2^{-i}\eps \le C\mathcal H_1(H_i)+2^{-i}\eps\le C\mathcal H(H_i)+2^{-i}\eps;
\]
see Remark \ref{rmk:capacities and Hausdorff contents}.

Then define for each $i\in\N$
\[
G_i:=
\left\{x\in X\setminus \bigcup_{j=1}^i U_j:\,\exists r\in(0,r_i)\ \ \textrm{such that }  
\ r\frac{\mathcal H(A\cap B(x,r))}{\mu(B(x,r))}\ge \frac 1i\right\}.
\]
Fix $i\in\N$. From the definition of $G_i$ we obtain a covering $\{B(x,r(x))\}_{x\in G_i}$ of $G_i$, and by the $5$-covering 
theorem, we can extract a countable collection of disjoint balls $\{B(x_k,r_k)\}_{k\in\N}$ such that the balls $B(x_k,5r_k)$ cover $G_i$.
Thus by Remark \ref{rmk:capacities and Hausdorff contents},
\begin{align*}
 \capa_1(G_i)
&\le C\mathcal H_{1}(G_i)
\le C\sum_{k\in\N}\frac{\mu(B(x_k,5r_k))}{5r_k}
\le C\sum_{k\in\N}\frac{\mu(B(x_k,r_k))}{r_k}\\
&\le C i\sum_{k\in\N}\mathcal H(A\cap B(x_k,r_k))
\le C i\mathcal H\left(A\setminus \bigcup_{j=1}^i H_j\right)\le Ci2^{-i}\eps.
\end{align*}

Thus 
\begin{align*}
\capa_1\left(\bigcup_{i\in\N}U_i\cup G_i\right)
&\le \sum_{i\in\N}\capa_1(U_i)+\sum_{i\in\N}\capa_1(G_i)\\
&\le \sum_{i\in\N}\left(C\mathcal H(H_i)+2^{-i}\eps\right)+C\sum_{i\in\N}i 2^{-i}\eps\\
&\le C\mathcal H(A)+C\eps.
\end{align*}
For $x\in X\setminus \bigcup_{i\in\N}U_i\cup G_i$, if $0<r<r_j$, then
\[
r\frac{\mathcal H(A\cap B(x,r))}{\mu(B(x,r))}< \frac 1j.
\]
Finally, since
\[
\capa_1\left(A\setminus \bigcup_{i\in\N}U_i\right)\le
C\mathcal H\left(A\setminus \bigcup_{i\in\N}U_i\right)=0,
\]
we can choose an open set $V\supset A\setminus \bigcup_{i\in\N}U_i$ with
$\capa_1(V)<\eps$,
and then we can take $U:=\bigcup_{i\in\N}U_i\cup G_i\cup V$.
\end{proof}

\begin{lemma}\label{lem:uniform convergence of G}
Let $G\subset X$ and $\eps>0$. Then there exists an open set $U\supset G$ with $\capa_1(U)\le C\capa_1(G)+\eps$ such that
\[
\frac{\mu(B(x,r)\cap G)}{\mu(B(x,r))}\to 0\qquad\textrm{as }r\to 0
\]
uniformly for $x\in X\setminus U$.
\end{lemma}

\begin{proof}
We can assume that $\capa_1(G)<\infty$. By Remark \ref{rmk:capacities and Hausdorff contents}, we have
\[
\mathcal H_{\diam(X)/10}(G)\le C\capa_1(G).
\]
(Of course we may have $\diam(X)/10=\infty$.)
Thus we can pick a covering $\{B(x_k,r_k)\}_{k\in\N}$ of $G$ with $r_k\le \diam(X)/10$ for all $k\in\N$ and
\begin{equation}\label{eq:capacity cover of G}
\sum_{k\in\N}\frac{\mu(B(x_k,r_k))}{r_k}\le C\capa_1(G)+\eps.
\end{equation}
For any fixed $k\in\N$, consider the following three properties.
\begin{enumerate}
\item By \cite[Lemma 6.2]{KKST} (or more precisely its proof) we have
\[
\frac{\mu(B(x_k,\rho))}{\rho}\le C P(B(x_k,\rho),X)
\]
for every $\rho\in [r_k,2r_k]$; note that here we need the fact that $r_k\le \diam(X)/10$.
\item By applying the $\BV$ coarea formula \eqref{eq:coarea} with $u(y)=\dist(y,x_k)$
 and $\Omega=B(x_k,2r_k)$, we  have $P(B(x_k,\rho),X)<\infty$ for almost every $\rho\in [r_k,2r_k]$.
\item By applying the coarea inequality given in Lemma \ref{lem:coarea inequality}
with $w(y)=\dist(y,x_k)$ and $U=B(x_k,2r_k)$, we conclude that there exists $T\subset [r_k,2r_k]$ with $\mathcal L^1(T)\ge r_k/2$
such that
\[
\mathcal H(\partial B(x_k,\rho))
\le 2C_{\textrm{co}}\frac{\mu(B(x_k,2r_k))}{r_k}
\le 2C_{\textrm{co}}C_d\frac{\mu(B(x_k,\rho))}{\rho}
\]
for every $\rho\in T$.
\end{enumerate}
Thus for each $k\in\N$ we can find a radius $\widetilde{r}_k\in [r_k,2r_k]$ with
\begin{equation}\label{eq:Hausdorff measures of boundaries of balls}
\begin{split}
\mathcal H(\partial B(x_k,\widetilde{r}_k))
&\le C\frac{\mu(B(x_k,\widetilde{r}_k))}{\widetilde{r}_k}\\
&\le C P(B(x_k,\widetilde{r}_k),X)\le C\mathcal H(\partial B(x_k,\widetilde{r}_k)),
\end{split}
\end{equation}
where the last inequality follows from \eqref{eq:def of theta}.
Let $A:=\bigcup_{k\in\N}\partial B(x_k,\widetilde{r}_k)$,
so that by the above and \eqref{eq:capacity cover of G},
\begin{align*}
\mathcal H(A)
\le \sum_{k\in\N}\mathcal H(\partial B(x_k,\widetilde{r}_k))
&\le C\sum_{k\in\N}\frac{\mu(B(x_k,\widetilde{r}_k))}{\widetilde{r}_k}\\
&\le C\sum_{k\in\N}\frac{\mu(B(x_k,r_k))}{r_k}\le C\capa_1(G)+C\eps.
\end{align*}
Note that if for any given ball $B(x,r)$ we have $\mathcal H(\partial B(x,r))<\infty$, then for any $y\in X$ we have $\mathcal H(\partial B(x,r)\cap \partial B(y,s))=0$ for almost every $s>0$.
Thus we can pick the radii $\widetilde{r}_k$ recursively in such a way that we also have $\mathcal H(\partial B(x_k,\widetilde{r}_k)\cap\partial B(x_l,\widetilde{r}_l))=0$ whenever $k\neq l$.

Then take a set $U\supset A$ with
\[
\capa_1(U)\le C\mathcal H(A)+\eps\le C\capa_1(G)+C\eps.
\]
as given by Lemma \ref{lem:uniform convergence}.
We can assume
that also $U\supset \bigcup_{k\in\N}B(x_k,2\widetilde{r}_k)$, since by Remark \ref{rmk:capacities and Hausdorff contents} and \eqref{eq:capacity cover of G},
\begin{align*}
\capa_1\left(\bigcup_{k\in\N}B(x_k,2\widetilde{r}_k)\right)
&\le C\mathcal H_{\diam(X)/5}\left(\bigcup_{k\in\N}B(x_k,2\widetilde{r}_k)\right)\\
&\le C\sum_{k\in\N}\frac{\mu(B(x_k,2 \widetilde{r}_k))}{2\widetilde{r}_k}\\
&\le C\sum_{k\in\N}\frac{\mu(B(x_k,r_k))}{r_k}\\
&\le C\capa_1(G)+C\eps.
\end{align*}
Let $x\in X\setminus U$. If for $r>0$ we have $B(x_k,\widetilde{r}_k)\cap B(x,r)\neq \emptyset$, then since $B(x_k,2\widetilde{r}_k)\subset U$, we have $\widetilde{r}_k\le \dist(B(x_k,\widetilde{r}_k),X\setminus U)\le r$.
Denoting $\widetilde{B}_k:=B(x_k,\widetilde{r}_k)$, we have
\begin{align*}
\limsup_{r\to 0}\frac{\mu(B(x,r)\cap G )}{\mu(B(x,r))}
&\le \frac{\mu\left(B(x,r)\cap \bigcup_{k\in\N}\widetilde{B}_k \right)}{\mu(B(x,r))}\\
&\le r\frac{\sum_{\widetilde{B}_k\cap  B(x,r)\neq \emptyset}\mu(\widetilde{B}_k)/\widetilde{r}_k}{\mu(B(x,r))}\\
&\overset{\eqref{eq:Hausdorff measures of boundaries of balls}}{\le} Cr\frac{\sum_{\widetilde{B}_k\cap  B(x,r)\neq \emptyset}\mathcal H(\partial \widetilde{B}_k)}{\mu(B(x,r))}\\
&\le Cr\frac{\mathcal H(A\cap B(x,3r))}{\mu(B(x,r))}\\
&\to 0
\end{align*}
uniformly as $r\to 0$ by Lemma \ref{lem:uniform convergence}.
\end{proof}

The following lemma can be proved by very similar methods as those used above.

\begin{lemma}[{\cite[Lemma 3.1]{L}}]\label{lem:perimeter of G}
For any $G\subset X$, we can find an open set $U\supset G$ with
$\capa_1(U)\le C\capa_1(G)$ and $P(U,X)\le C\capa_1(G)$.
\end{lemma}

The next lemma gives a standard fact about the relationship between Hausdorff content and measure.

\begin{lemma}\label{lem:Hausdorff content and measure}
Let $A\subset X$ and $R>0$. If $\mathcal H_R(A)=0$, then $\mathcal H(A)=0$.
\end{lemma}
Note that the converse implication is trivial.
In \cite[Lemma 7.9]{KKST3} it was shown that we  have the above even for $R=\infty$, under the additional assumption that the space is \emph{$1$-hyperbolic}, but we do not 
need to consider this assumption in this paper.

\begin{proof}
We can assume that $A$ is bounded, and so $A\subset B(x_0,R_0)$ for some
$x_0\in X$ and $R_0\ge R$.
Fix $\eps>0$. By the fact that $\mathcal H_R(A)=0$, we can find a covering $\{B(x_j,r_j)\}_{j\in\N}$ of $A$ such that $r_j\le R$ for all $j\in\N$ and
\[
\sum_{j\in\N} \frac{\mu(B(x_j,r_j))}{r_j}<\eps.
\]
We can also assume that $B(x_j,r_j)\cap A\neq \emptyset$ for all $j\in\N$, and so $x_j\in B(x_0,2R_0)$ for all $j\in\N$. Note that we can choose $Q>1$ in \eqref{eq:homogenous dimension}. Then for each $j\in\N$ we have
\[
\frac{\mu(B(x_0,2R_0))}{(2R_0)^Q}r_j^{Q-1}\le C\frac{\mu(B(x_j,r_j))}{r_j}<C\eps,
\]
so that
\[
r_j<\left(C\eps\frac{(2R_0)^Q}{\mu(B(x_0,2R_0))}\right)^{1/(Q-1)}=:\delta_{\eps},
\]
so in fact we have
\[
\mathcal H_{\delta_{\eps}}(A)\le \sum_{j\in\N} \frac{\mu(B(x_j,r_j))}{r_j}<\eps.
\]
Here $\delta_{\eps}\to 0$
as $\eps\to 0$. Thus $\mathcal H(A)=\lim_{\eps\to 0}\mathcal H_{\delta_{\eps}}(A)=0$.
\end{proof}

The following lemma is well known e.g. in the Euclidean setting. We will only use it in the special case of sets of finite perimeter, but we give the standard proof for more general $\BV$ functions.

\begin{lemma}\label{lem:variation measure and Hausdorff content}
Let $\Omega\subset X$ be an open set,
let $u\in L^1_{\loc}(\Omega)$ with $\Vert Du\Vert(\Omega)<\infty$, and let $R>0$. Then for every $\eps>0$ there exists $\delta>0$ such that if $A\subset \Omega$ with $\mathcal H_{R}(A)<\delta$, then $\Vert Du\Vert(A)<\eps$.
\end{lemma}
\begin{proof}
By the $\BV$ coarea formula \eqref{eq:coarea}, for almost every $t\in \R$ we have $P(\{u>t\},\Omega)<\infty$, and by \eqref{eq:def of theta} we have
\[
\mathcal H(\partial^*\{u>t\}\cap\Omega)\le CP(\{u>t\},\Omega)
\]
for such $t$. Fix one such $t\in\R$. Assume that there exists $\delta>0$ and a sequence  of Borel sets $A_i$, $i\in\N$, such that $\mathcal H_{R}(A_i)\le 2^{-i}$ but $\mathcal H|_{\partial^*\{u>t\}\cap\Omega}(A_i)\ge \delta$. Then defining
\[
A:=\bigcap_{i\in\N}\bigcup_{j\ge i}A_j,
\]
we have $\mathcal H_{R}(A)=0$ but $\mathcal H(A)\ge \delta$, a contradiction
by Lemma \ref{lem:Hausdorff content and measure}. Thus for 
almost every $t\in\R$, $\mathcal H|_{\partial^*\{u>t\}\cap\Omega}(A)\to 0$ if $\mathcal H_{R}(A)\to 0$, for $A$ Borel.
 
By the coarea formula \eqref{eq:coarea},
\[
\Vert Du\Vert(A)=\int_{\R}P(\{u>t\},A)\,dt
\]
for any Borel set $A\subset \Omega$.
Here we have by \eqref{eq:def of theta} again that $P(\{u>t\},A)\le C\mathcal H(\partial^*\{u>t\}\cap A)$ for almost every $t\in\R$.
By using Lebesgue's dominated convergence theorem, with the majorant  function $t\mapsto P(\{u>t\},\Omega)$, we get $\Vert Du\Vert(A)\to 0$ if $\mathcal H_{R}(A)\to 0$, with $A$ Borel. The result for general sets $A\subset \Omega$ follows by approximation.
\end{proof}

\begin{lemma}\label{lem:variation measure and capacity}
Let $\Omega\subset X$ be an open set and
let $u\in L^1_{\loc}(\Omega)$ with $\Vert Du\Vert(\Omega)<\infty$. Then for every $\eps>0$ there exists $\delta>0$ such that if $A\subset \Omega$ with $\capa_1 (A)<\delta$, then $\Vert Du\Vert(A)<\eps$.
\end{lemma}
\begin{proof}
Combine Remark \ref{rmk:capacities and Hausdorff contents} and Lemma \ref{lem:variation measure and Hausdorff content}.
\end{proof}

\section{Quasicontinuity}

In this section we present and slightly generalize the quasicontinuity-type result for $\BV$ functions given in \cite{LaSh}.

In the Euclidean setting, results on the fine properties of $\BV$ functions can be formulated in terms of the lower and upper approximate limits $u^{\wedge}$ and $u^{\vee}$
given in \eqref{eq:lower approximate limit} and \eqref{eq:upper approximate limit}. In the metric setting, we need to consider more than two jump values.
Recall the definition of the number $\gamma$ from \eqref{eq:definition of gamma}.
Then we define the functions $u^l$, $l=1,\ldots,n:=\lfloor 1/\gamma\rfloor$, as 
follows: $u^1:=u^{\wedge}$, $u^n:=u^{\vee}$, 
and for $l=2,\ldots,n-1$ we define inductively
\begin{equation}\label{eq:definition of n jump values}
u^{l}(x):=\sup\left\{t\in\overline{\R}:\,\lim_{r\to 0}\frac{\mu(B(x,r)\cap \{u^{l-1}(x)+\delta<u<t\})}{\mu(B(x,r))}=0\ \ \forall\, \delta>0\right\}
\end{equation}
provided $u^{l-1}(x)<u^\vee(x)$, and otherwise we set 
$u^l(x)=u^{\vee}(x)$. 
It can be shown that each $u^l$ is a Borel function,
and $u^{\wedge}=u^1\le \ldots \le u^n = u^{\vee}$.

We have the following notion of quasicontinuity for $\BV$ functions.

\begin{theorem}[{\cite[Theorem~1.1]{LaSh}}]\label{thm:quasicontinuity}
Let $u\in\BV(X)$ and let $\eps>0$. Then there exists an open set $G\subset X$ with $\capa_1(G)<\eps$ such that if $y_k\to x$ 
with $y_k,x\in X\setminus G$, then 
\[
\min_{l_2\in \{1,\ldots,n\}} |u^{l_1}(y_k)-u^{l_2}(x)|\to 0
\]
for each $l_1=1,\ldots,n$.
\end{theorem}

First we give a local version of this result, as follows.

\begin{corollary}\label{cor:quasicontinuity local case}
Let $\Omega\subset X$ be an open set,
let $u\in\BV_{\loc}(\Omega)$, and let $\eps>0$. Then there exists an open set $G\subset \Omega$ with $\capa_1(G)<\eps$ such that if $y_k\to x$ 
with $y_k,x\in \Omega\setminus G$, then 
\[
\min_{l_2\in \{1,\ldots,n\}} |u^{l_1}(y_k)-u^{l_2}(x)|\to 0
\]
for each $l_1=1,\ldots,n$.
\end{corollary}

\begin{proof}
Pick sets $\Omega_1\Subset \Omega_2\Subset \ldots$ with $\Omega=\bigcup_{j\in\N}\Omega_j$.
Also pick cutoff functions $\eta_j\in\Lip_c(\Omega_{j+1})$ with $0\le\eta_j\le 1$
and  $\eta_j=1$ in $\Omega_j$ for
each $j\in\N$. Denote the Lipschitz constants by $L_j$. Fix $j\in\N$. We have $u\in \BV(\Omega_{j+1})$, so that we find a sequence $\liploc(\Omega_{j+1})\ni u_i\to u$ in $L_{\loc}^1(\Omega_{j+1})$
with
\[
\lim_{i\to\infty}\int_{\Omega_{j+1}}g_{u_i}\,d\mu=\Vert Du\Vert(\Omega_{j+1}).
\]
Recall that $g_{u_i}$ denotes the minimal $1$-weak upper gradient of $u_i$.
Clearly $\eta_j u_i\in \Lip(X)$ with $\eta_j u_i\to \eta_j u$ in $L^1(X)$
as $i\to\infty$.
Thus by the definition of the total variation and by the Leibniz rule for Newton-Sobolev functions, see \cite[Theorem 2.15]{BB}, we have
\begin{align*}
\Vert D(\eta_j u)\Vert (X)
&\le \liminf_{i\to\infty}\int_{X} g_{\eta_j u_i}\,d\mu\\
&\le \liminf_{i\to\infty}\int_{X} g_{\eta_j} |u_i|+\eta_j g_{u_i} \,d\mu\\
&\le \limsup_{i\to\infty}\int_{\supp(\eta_j)} L_j |u_i|\,d\mu+\limsup_{i\to\infty}\int_{\Omega_{j+1}} g_{u_i} \,d\mu\\
&\le  L_j \Vert u\Vert_{L^1(\Omega_{j+1})}+\Vert Du\Vert(\Omega_{j+1})<\infty.
\end{align*}
Thus $u_j:=\eta_j u\in\BV(X)$ for each $j\in\N$, and so we can apply Theorem \ref{thm:quasicontinuity} to obtain open sets $G_j\subset X$ with $\capa_1(G_j)< 2^{-j}\eps$.
Defining $G:=\bigcup_{j\in\N}G_j\cap\Omega$, we have $\capa_1(G)< \eps$, and if
$y_k\to x$ with $y_k,x\in \Omega\setminus G$, then $y_k,x\in\Omega_j$ for some $j\in\N$
and thus for large enough $k\in\N$
\[
\min_{l_2\in \{1,\ldots,n\}} |u^{l_1}(y_k)-u^{l_2}(x)|=\min_{l_2\in \{1,\ldots,n\}}
|(u\eta_j)^{l_1}(y_k)-(u\eta_j)^{l_2}(x)|\to 0
\]
as $k\to\infty$
for each $l_1=1,\ldots,n$, by the fact that $y_k,x\notin G_j$.

\end{proof}

Recall the definitions of the measure theoretic interior and exterior $I_E$ and $O_E$
of a set $E\subset X$ from \eqref{eq:definition of measure theoretic interior} and \eqref{eq:definition of measure theoretic exterior}. Note that for $u=\ch_E$, we have $x\in I_E$ if and only if $u^{\wedge}(x)=u^{\vee}(x)=1$, $x\in O_E$ if and only if $u^{\wedge}(x)=u^{\vee}(x)=0$, and $x\in \partial^*E$ if and only if $u^{\wedge}(x)=0$ and $u^{\vee}(x)=1$.
Moreover, in this case $u^1=u^{\wedge}$ and $u^2=\ldots =u^n=u^{\vee}$.

In this paper we will only need the following notion of quasicontinuity for sets of finite perimeter, which is obtained by
applying Corollary \ref{cor:quasicontinuity local case} to $u=\ch_E$.

\begin{corollary}\label{cor:quasicontinuity for sets of finite perimeter}
Let $\Omega\subset X$ be an open set, let $E\subset X$ be a $\mu$-measurable set with
$P(E,\Omega)<\infty$, and let $\eps>0$. Then there exists an open set $G\subset \Omega$ with $\capa_1(G)<\eps$ such that if $y_k\to x$ 
with $y_k,x\in \Omega\setminus G$, then 
\[
\min\{|\ch_E^{\wedge}(y_k)-\ch_E^{\wedge}(x)|,\,|\ch_E^{\wedge}(y_k)-\ch_E^{\vee}(x)|\}\to 0
\]
and
\[
\min\{|\ch_E^{\vee}(y_k)-\ch_E^{\wedge}(x)|,\,|\ch_E^{\vee}(y_k)-\ch_E^{\vee}(x)|\}\to 0.
\]
\end{corollary}

For example, if $x\in O_E$, then $\ch_E^{\wedge}(x)=0=\ch_E^{\vee}(x)$ and necessarily $y_k\in O_E$ for sufficiently large $k\in\N$.

\section{Approximation of sets of finite perimeter}

In this section we prove our main result on the approximation of a set of finite perimeter
by more regular sets in the $\BV$ norm.

We will need to work with Whitney-type coverings of open sets.
For the construction of such coverings and their properties, see e.g.~\cite[Theorem 3.1]{BBS07}.
Given any  open set $U\subset X$ and a scale $R>0$, we can choose a Whitney-type covering 
$\{B_j=B(x_j,r_j)\}_{j=1}^{\infty}$ of $U$ such that
\begin{enumerate}
\item for each $j\in\N$,
\begin{equation}\label{eq:definition of radius of a Whitney ball}
r_j= \min\left\{\frac{\dist(x_j,X\setminus U)}{40\lambda},\,R\right\},
\end{equation}
\item for each $k\in\N$, the ball $10\lambda B_k$ meets at most $C=C(C_d,\lambda)$ 
balls $10\lambda B_j$ (that is, a bounded overlap property holds),
\item  
if $10\lambda B_j$ meets $10\lambda B_k$, then $r_j\le 2r_k$.
\end{enumerate}

Given such a covering of $U$, 
we can take a partition of unity $\{\phi_j\}_{j=1}^{\infty}$ subordinate to the covering, such that $0\le \phi_j\le 1$, each 
$\phi_j$ is a $C/r_j$-Lipschitz function, and $\supp(\phi_j)\subset 2B_j$ for each 
$j\in\N$ (see e.g. \cite[Theorem 3.4]{BBS07}). Finally, we can define a \emph{discrete convolution} $v$ of 
any $u\in L^1_{\loc}(U)$ with respect to the Whitney-type covering by
\begin{equation}\label{eq:definition of discrete convolution}
v:=\sum_{j=1}^{\infty}u_{B_j}\phi_j.
\end{equation}
In general, $v$ is locally Lipschitz in $U$, and hence belongs to $ L^1_{\loc}(U)$. 

We can ``mollify'' $\BV$ functions in open sets in the following manner. Recall the definition of the pointwise representative $\widetilde{u}$ from \eqref{eq:precise representative}.

\begin{theorem}\label{thm:mollifying in an open set}
Let $U\subset\Omega\subset X$ be open sets, and let $u\in L_{\loc}^1(X)$ with $\Vert Du\Vert(\Omega)<\infty$.
Then there exists a function $w\in L_{\loc}^1(X)$ with $\Vert Dw\Vert(\Omega)<\infty$ such that $w=u$ in $\Omega\setminus U$ and
$\widetilde{w}|_{U}\in N^{1,1}(U)\cap \liploc(U)$ with an
upper gradient $g$ satisfying $\Vert g\Vert_{L^1(U)}\le C\Vert Du\Vert(U)$.

The function $w$ is defined in $U$ as a limit of discrete convolutions of $u$ with respect to Whitney-type coverings of open sets
$U_1\subset U_2\subset \ldots\subset U$ with $U=\bigcup_{i\in\N}U_i$, at an arbitrary fixed scale $R>0$.
For $\mathcal H$-almost every $x\in \partial U$ we have
\begin{equation}\label{eq:boundary trace behavior for mollified function}
\frac{1}{\mu(B(x,r))}\int_{B(x,r)\cap U}|w-u|\,d\mu\to 0
\end{equation}
as $r\to 0$.
\end{theorem}
This is essentially \cite[Corollary~3.6]{LaSh}.
The last two sentences of the theorem are not part of \cite[Corollary~3.6]{LaSh}, but follow from
its proof. Moreover, in \cite[Corollary~3.6]{LaSh} we make the assumption $u\in\BV(\Omega)$, but the proof runs through almost verbatim for the slightly more general case presented here.

Now we give our main result.

\begin{theorem}\label{thm:approximation for sets of finite perimeter}
Let $\Omega\subset X$ be an open set, let $E\subset X$ be a $\mu$-measurable set with $P(E,\Omega)<\infty$, and let $\eps>0$. Then there exists a $\mu$-measurable set $F\subset X$ with
\[
\Vert \ch_F-\ch_E\Vert_{\BV(\Omega)}<\eps\quad\textrm{and}\quad \mathcal H(\Omega\cap \partial F\setminus \partial^*F)=0.
\]
\end{theorem}

\begin{proof}
Apply Corollary \ref{cor:quasicontinuity for sets of finite perimeter} to obtain a set
$G\subset \Omega$ with $\capa_1(G)<\eps$, and then apply Lemma
\ref{lem:uniform convergence of G} to obtain an open set $U\subset \Omega$ with
$U\supset G$ such that
\[
\frac{\mu(B(x,r)\cap G )}{\mu(B(x,r))}\to 0\qquad\textrm{as }r\to 0
\]
uniformly for $x\in \Omega\setminus U$.
By Lemma \ref{lem:variation measure and capacity} we can also assume that
\begin{equation}\label{eq:smallness of perimeter in U}
\Vert D\ch_E\Vert(U)< \eps.
\end{equation}

In the following, we ``mollify'' $\ch_E$ in the set $U$ and then define $F$ as a super-level set of the mollified function. First, apply Theorem \ref{thm:mollifying in an open set} with $u=\ch_E$ and at the scale $R=1$ to obtain a function $w\in L^1_{\loc}(X)$ with
$\Vert Dw\Vert(\Omega)<\infty$ and $\widetilde{w}\in\liploc(U)$.

Fix $x\in\Omega\cap\partial U$ with $x\in O_E$. By Corollary \ref{cor:quasicontinuity for sets of finite perimeter}, there exists $\delta\in (0,1)$ such that $B(x,\delta)\subset\Omega$ and
\begin{equation}\label{eq:using quasicontinuity}
y\in O_E\ \ \ \textrm{for all}\ y\in B(x,\delta)\setminus G.
\end{equation}
By making $\delta$ smaller, if necessary, by Lemma \ref{lem:uniform convergence of G} we also have
\begin{equation}\label{eq:using uniform convergence for G}
\frac{\mu(B(z,r)\cap G)}{\mu(B(z,r))}<\frac{1}{4 C_d^{\lceil\log_2 ( 200\lambda)\rceil}}
\end{equation}
for all $z\in X\setminus U$ and $r\in (0,\delta)$.
Here $\lceil a\rceil$ is the smallest integer at least $a\in\R$.

Fix $y\in B(x,\delta/4)\cap U$.
Recall that $w$ is defined in $U$ as a limit of discrete convolutions of $u$
with respect to Whitney-type coverings $\{B_j^i=B(x_j^i,r_j^i)\}_{j\in\N}$ 
of sets $U_i\subset U$ at scale $R=1$.
Since $U=\bigcup_{i\in\N}U_i$,
we can fix a sufficiently large $i\in\N$ such that
\[
\dist(y,X\setminus U_i)\ge \dist(y,X\setminus U)/2.
\]
Suppose that $y\in B(x_j^i,2 r_j^i)$. It is easy to see that $B(x_j^i,2 r_j^i)\subset B(x,\delta)$.
Then 
\begin{align*}
r_j^i= \min\left\{\frac{\dist(x_j^i,X\setminus U_i)}{40\lambda},\,R\right\}
&= \frac{\dist(x_j^i,X\setminus U_i)}{40\lambda}\\
&\ge \frac{\dist(y,X\setminus U_i)-2r_j^i}{40\lambda}\\
&\ge \frac{\dist(y,X\setminus U)}{80\lambda}-\frac{r_j^i}{20\lambda}.
\end{align*}
Thus
\[
r_j^i\ge \frac{\dist(y,X\setminus U)}{90\lambda}.
\]
Since $B(x,\delta)\subset \Omega$, there is $z\in \Omega\setminus U$ with
$d(y,z)=\dist(y,X\setminus U)$.
Then $B(z,2d(y,z))\subset 200 \lambda B_j^i$, so by the doubling property of the measure
\[
\mu(B(z,2d(y,z)))\le C_d^{\lceil\log_2(200\lambda)\rceil}\mu(B_j^i).
\]
Moreover, 
\[
d(y,z)=\dist(y,X\setminus U)\le d(y,x)<\delta/4,
\]
and thus $d(x,z)<\delta/2$. Hence
$B(z,2d(y,z))\subset B(x,\delta)$,
so that
\[
2B_j^i\setminus O_E\subset B(z,2d(y,z))\setminus O_E\subset B(z,2d(y,z))\cap G
\]
by \eqref{eq:using quasicontinuity}.
Using this and \eqref{eq:using uniform convergence for G}, we obtain
\begin{align*}
u_{B^i_j}
=\frac{\mu(E\cap B^i_j)}{\mu(B^i_j)}
&\le \frac{C_d^{\lceil\log_2(200\lambda)\rceil}}{\mu(B(z,2d(y,z)))}\mu(B(z,2d(y,z))\cap E)\\
&\le \frac{C_d^{\lceil\log_2(200\lambda)\rceil}}{\mu(B(z,2d(y,z)))}\mu(B(z,2d(y,z))\cap G)\\
&\le \frac{1}{4}.
\end{align*}
For each $i\in\N$, let $w_i$ be the discrete convolution of $u$ in $U_i$ 
with respect to the Whitney-type covering $\{B_j^i\}_{j\in\N}$.
Recalling the definition of a discrete convolution from \eqref{eq:definition of discrete convolution}, we have for suitable Lipschitz functions $\phi_j^i$
\[
w_i(y)=\sum_{j\in\N}u_{B_j^i}\phi_j^i(y) \le \frac{1}{4}\sum_{j\in\N}\phi_j^i(y)=\frac{1}{4}.
\]
According to Theorem \ref{thm:mollifying in an open set}, the quantity $\widetilde{w}(y)$ is defined as the limit of $w_i(y)$ as $i\to\infty$, so we have $\widetilde{w}(y)\le 1/4$.
Since $y\in B(x,\delta/4)\cap U$ was arbitrary, we have $\widetilde{w}\le 1/4$ in $B(x,\delta/4)\cap U$.
Similarly, for any $x\in\Omega\cap \partial U\cap I_E$ there exists some $r>0$ such that $\widetilde{w}\ge 3/4$ in $B(x,r)\cap U$.

By the $\BV$ coarea formula \eqref{eq:coarea}, we can find a set $T\subset (1/4,3/4)$ with $\mathcal L^1(T)\ge 1/4$ such that for all $t\in T$,
\begin{equation}\label{eq:estimate for boundary of level set}
\Vert D\ch_{\{w>t\}}\Vert(U)\le 4\Vert Dw\Vert(U)\le C\Vert D\ch_E\Vert(U),
\end{equation}
where the last inequality follows from Theorem 
\ref{thm:mollifying in an open set}.
By \eqref{eq:boundary trace behavior for mollified function}, there exists $N\subset \partial U$ with $\mathcal H(N)=0$ such that for every $x\in \partial U\setminus N$, we have
\[
\frac{1}{\mu(B(x,r))}\int_{B(x,r)\cap U}|w-\ch_E|\,d\mu \to 0
\]
as $r\to 0$.
For any fixed $t\in (0,1)$, this implies
\begin{equation}\label{eq:Lebesgue property at boundary of U}
\begin{split}
&\frac{1}{\mu(B(x,r))}\int_{B(x,r)\cap U}|\ch_{\{w>t\}}-\ch_E|\,d\mu\\
&\qquad\quad \le \frac{1}{\min\{t,1-t\}}\frac{1}{\mu(B(x,r))}\int_{B(x,r)\cap U}|w-\ch_E|\,d\mu\to 0
\end{split}
\end{equation}
as $r\to 0$.

Again by the $\BV$ coarea formula \eqref{eq:coarea}, for almost every $t\in (0,1)$, setting
\[
F_t:=(I_E\cap\Omega\setminus U)\cup(\{\widetilde{w}>t\}\cap U),
\]
so that $F_t=\{w>t\}\cap\Omega$ as $\mu$-equivalence classes,
we have $P(F_t,\Omega)<\infty$. By \eqref{eq:Lebesgue property at boundary of U}, for
every $x\in\Omega\setminus (U\cup N)$ and for all $s\neq 0$, we have $x\notin \partial^*\{\ch_{F_t}-\ch_E>s\}$. Thus by the $\BV$ coarea formula \eqref{eq:coarea} and \eqref{eq:def of theta}, for almost every $t\in (0,1)$
\begin{align*}
\Vert D(\ch_{F_t}-\ch_E)\Vert(\Omega\setminus U)
&=\int_{-\infty}^{\infty}P(\{\ch_{F_t}-\ch_E>s\},\Omega\setminus U)\,ds\\
&\le C\int_{-\infty}^{\infty}\mathcal H(\partial^*\{\ch_{F_t}-\ch_E>s\}\cap (\Omega\setminus U))\,ds\\
&=0.
\end{align*}
By using this and \eqref{eq:estimate for boundary of level set}, we have for almost every $t\in T$
\begin{align*}
\Vert D(\ch_{F_t}-\ch_E)\Vert(\Omega)
&=\Vert D(\ch_{F_t}-\ch_E)\Vert(U)\\
&\overset{\eqref{eq:subadditivity}}{\le} \Vert D\ch_{F_t}\Vert(U)+ \Vert D\ch_E\Vert(U)\\
&\le C\Vert D\ch_E\Vert(U)+ \Vert D\ch_E\Vert(U)\\
&< C\eps
\end{align*}
by \eqref{eq:smallness of perimeter in U}, and also
\begin{equation}\label{eq:null difference between topological and meas theor bdry}
\mathcal H(U\cap\partial \{\widetilde{w}>t\}\setminus \partial^*\{\widetilde{w}>t\})=0
\end{equation}
by Proposition \ref{prop:null difference between topological and meas theor bdry}.
We fix one such $t$ and define $F:=F_t$. Since
\[
\Vert \ch_F-\ch_E\Vert_{L^1(\Omega)}\le \mu(U)\le \capa_1(U)< \eps,
\]
we have $\Vert \ch_F-\ch_E\Vert_{\BV(\Omega)}< C\eps$ and one claim of the theorem is proved.

From Corollary \ref{cor:quasicontinuity for sets of finite perimeter} we know that if
$x\in\partial I_E\cap \Omega\setminus \overline{U}$, then $x\in\partial^*E$. Thus
from the definition of $F$ it follows that
\[
\partial F\cap\Omega\setminus \overline{U}=\partial I_E\cap\Omega\setminus \overline{U}=\partial^* E\cap\Omega\setminus \overline{U}=\partial^*F\cap\Omega\setminus \overline{U}.
\]
If $x\in \Omega\cap \partial U\cap O_E$,
the previously proved fact that $\widetilde{w}\le 1/4$ in $B(x,r)\cap U$ for some $r>0$ implies that
\[
\ch_{\{\widetilde{w}>t\}}(y)=0\quad \textrm{for all}\ y\in B(x,r)\cap U
\]
for any $t\in (1/4,3/4)$.
Combining this with \eqref{eq:using quasicontinuity}, we conclude that $x$ is an exterior point of $F$.
Analogously, if $x\in \Omega\cap \partial U\cap I_E$, then $x$ is an interior point of $F$. If $x\in \Omega\cap\partial U\cap\partial^*E\setminus N$, then $x\in \partial U\cap \partial^*F$ by \eqref{eq:Lebesgue property at boundary of U}. In total,
\[
\partial F\setminus \partial^*F\subset (U\cap \partial F\setminus \partial^*F)\cup N
= (U\cap \partial\{\widetilde{w}>t\}\setminus \partial^*\{\widetilde{w}>t\})\cup N.
\]
Hence
\[
\mathcal H(\partial F\setminus \partial^*F)\le \mathcal H(U\cap \partial\{\widetilde{w}>t\}\setminus \partial^*\{\widetilde{w}>t\})=0
\]
by \eqref{eq:null difference between topological and meas theor bdry}.
\end{proof}

\begin{example}\label{ex:enlarged rationals}
A standard example illustrating how badly behaved a set of finite perimeter can be is given by the so-called \emph{enlarged rationals}.
Consider the Euclidean space $\R^2$ equipped with the Lebesgue measure $\mathcal L^2$.
Let $\{q_i\}_{i\in\N}$ be an enumeration of $\mathbb{Q}\times\mathbb{Q}\subset\R^2$,
and define
\[
E:=\bigcup_{i\in\N} B(q_i,2^{-i}).
\]
Clearly $\mathcal L^2(E)\le \pi$. By the lower semicontinuity and subadditivity of perimeter, see \eqref{eq:Caccioppoli sets form an algebra}, we can estimate
\[
P(E,\R^2)\le \sum_{i=1}^{\infty} P(B(q_i,2^{-i}),\R^2)\le 2\pi \sum_{i=1}^{\infty}2^{-i},
\]
so that $P(E,\R^2)<\infty$, and then also $\mathcal H(\partial^*E)<\infty$.
On the other hand, $\partial E=\R^2\setminus E$, so that $\mathcal L^2(\partial E)=\infty$ and in particular $\mathcal H^1(\partial E)=\infty=\mathcal H(\partial E)$ (where $\mathcal H^1$ is the $1$-dimensional Hausdorff measure, which is comparable to the codimension $1$
Hausdorff measure $\mathcal H$). However, we can define the set $F\subset \R^2$ of Theorem \ref{thm:approximation for sets of finite perimeter} as
\[
F:=\bigcup_{i=1}^N B(q_i,2^{-i})
\]
for $N\in\N$ sufficiently large. It can then be shown that $\Vert \ch_F-\ch_E\Vert_{\BV(\R^2)}\to 0$ as $N\to\infty$, and that $\mathcal H(\partial F\setminus \partial^*F)=0$. By slightly modifying the set $F$ near the intersections of the spheres $\partial B(q_i,2^{-i})$, if necessary, we can even ensure that $\partial F=\partial^*F$.
\end{example}

\begin{openproblem}
In Theorem \ref{thm:approximation for sets of finite perimeter}, is it possible to obtain
 $\partial F\cap\Omega= \partial^*F\cap\Omega$?
\end{openproblem}

If the answer is yes, note that $\inte(F)\cap\Omega=I_F\cap\Omega$, and thus in $\Omega$, 
$\ch_{F}^{\wedge}=\ch_{I_F}$ is a lower semicontinuous function. Similarly, 
in $\Omega$, $\ch_F^{\vee}=\ch_{I_F\cup \partial^*F}=\ch_{\overline{F}}$ is then an upper semicontinuous function.

Note also that it follows from the proof of Theorem \ref{thm:approximation for sets of finite perimeter} that $\ch_F^{\wedge}$ and $\ch_E^{\wedge}$ can differ only in the set $U\cup N$, where $N$ is the $\mathcal H $-negligible set defined before \eqref{eq:Lebesgue property at boundary of U}. Thus we have
\[
\capa_1(\{\ch_F^{\wedge}\neq \ch_E^{\wedge}\})<\eps\qquad\textrm{and similarly}\qquad
\capa_1(\{\ch_F^{\vee}\neq \ch_E^{\vee}\})<\eps.
\]
For a more general $\BV$ function, we can now ask the following.

\begin{openproblem}
Let $\Omega\subset X$ be an open set, let $u\in\BV_{\loc}(\Omega)$, and let $\eps>0$. Can we find a function $v\in\BV_{\loc}(\Omega)$ with $\Vert v-u\Vert_{\BV(\Omega)}<\eps$,
\[
\capa_1(\{v^{\wedge}\neq u^{\wedge}\})<\eps,\qquad
\capa_1(\{v^{\vee}\neq u^{\vee}\})<\eps,
\]
and such that $v^{\wedge}$ is lower semicontinuous and $v^{\vee}$ is upper semicontinuous?
\end{openproblem}

%\begin{example}
%In Theorem \ref{thm:approximation for sets of finite perimeter}, we have in particular that
%\begin{equation}\label{eq:symmetric difference of boundaries}
%\mathcal H(\partial F\Delta\partial^*E)\le \mathcal H(\partial F\setminus \partial^*F)+\mathcal H(\partial F^*\Delta\partial^*E)\le \eps+C\Vert D(\ch_E-\ch_F)\Vert(X)\le 2\eps.
%\end{equation}
%\end{example}

\end{document}